\documentclass[12pt]{article}
\usepackage{amsmath,amsthm,verbatim,amssymb,amsfonts,amscd,diagrams, graphics, mathrsfs,hyperref}
\theoremstyle{plain}
\newtheorem{theorem}{Theorem}
\newtheorem*{intro-theorem}{Theorem}

\newtheorem*{intro-corollary}{Corollary}
\newtheorem*{lemma}{Lemma}
\newtheorem*{proposition}{Proposition}

\theoremstyle{definition}

\theoremstyle{remark}
\newtheorem*{remark}{Remark}

\newarrow{ul}---->
\newarrow{Backwards}<----

\newcommand{\Z}{\mathbb{Z}}
\newcommand{\Q}{\mathbb{Q}}
\newcommand{\R}{\mathbb{R}}

\newcommand{\C}{\mathbb{C}}

\newcommand{\bd}{\partial}

\renewcommand{\char}{\text{char}}

\newcommand{\mc}[1]{\mathcal{#1}}

\newcommand{\Hom}{\text{Hom}}

\newcommand{\im}{\text{im}}

\newarrow{Onto}----{>>}
\newarrow{Equals}=====
\newarrow{Into}C--->

\makeatletter
\def\blfootnote{\xdef\@thefnmark{}\@footnotetext}
\makeatother
\begin{document}

\title{{K}-{Witt} bordism in characteristic 2}
\author{Greg Friedman\footnote{This work was partially supported by a grant from the Simons Foundation (\#209127 to Greg Friedman)}}

\date{August 13, 2012}

\maketitle
\begin{abstract}
This note provides a computation of the bordism groups of $K$-Witt spaces for fields $K$ with characteristic $2$. We provide a complete computation for the unoriented bordism groups. For the oriented bordism groups, a nearly complete computation is provided as well a discussion of the difficulty of resolving a remaining ambiguity in dimensions equivalent to $2\mod 4$. This corrects an error in the $char(K)=2$ case of the author's prior computation of the bordism groups of $K$-Witt spaces for an arbitrary  field $K$. 
\end{abstract}
\blfootnote{\noindent\textbf{2000 Mathematics Subject Classification: }  55N33, 57Q20, 57N80 

\textbf{Keywords:} intersection homology,  Witt bordism, Witt space}

In \cite{GBF21}, an $n$-dimensional $K$-Witt space, for a field $K$, is defined\footnote{There is a minor error in \cite{GBF21} in that Witt spaces are stated to be irreducible, meaning that there is only a single top dimensional stratum. In general, this should not be part of the definition of a $K$-Witt space; cf. \cite{Si83}. However, as every $K$-Witt space of dimension $>0$ is bordant to an irreducible $K$-Witt space (see \cite[page 1099]{Si83}), this error does not affect the bordism group computations of \cite{GBF21}. It is not true that every $0$-dimensional $K$-Witt space is bordant to an irreducible $K$-Witt space, but in this dimension the computations all reduce to the manifold theory and the computations given for this dimension in \cite{GBF21} are also correct if one removes irreducibility from the definition.}  to be  an oriented compact  $n$-dimensional PL stratified pseudomanifold $X$ satisfying the $K$-Witt condition that the lower-middle perversity intersection homology group 
$I^{\bar m}H_k(L;K)$ is $0$ for each link $L^{2k}$ of each stratum of $X$ of dimension $n-2k-1$, $k>0$. Following the definition of stratified pseudomanifold in \cite{GM1}, $X$ does not possess codimension one strata. Orientability is determined by the orientability of the top (regular) strata. This definition generalizes Siegel's  definition in \cite{Si83} of $\Q$-Witt spaces (called there simply ``Witt spaces''). The   motivation for this definition is that such spaces possess intersection homology Poincar\'e duality $I^{\bar m}H_i(X;K)\cong \Hom(I^{\bar m}H_{n-i}(X;K),K)$. 

The author's paper \cite{GBF21} concerns $K$-Witt spaces and, in particular, a computation of the bordism theory $\Omega_*^{K-\text{Witt}}$ of such spaces. However, there is an error in \cite{GBF21} in the computation of the coefficient groups $\Omega_{4k+2}^{K-\text{Witt}}$ when $\char(K)=2$.

It is claimed in \cite{GBF21} that $\Omega_{4k+2}^{K-\text{Witt}}=0$. When $\char(K)> 2$, the null-bordism of a $4k+2$ dimensional $K$-Witt space $X$ is established in \cite{GBF21} by following Siegel's computation \cite{Si83} for $\Q$-Witt spaces by first performing a surgery to make the space irreducible and then
performing a sequence of singular surgeries  to obtain a space $X'$ such that  $I^{\bar m}H_{2k+1}(X';K)=0$. The $K$-Witt null-bordism of $X$ is the union of the trace of the surgeries from $X$ to $X'$ with the closed cone $\bar cX'$. One performs the singular surgeries on elements  $[z]\in I^{\bar m}H_{2k+1}(X;K)$ such that $[z]\cdot [z]=0$, where $\cdot$ denotes the Goresky-MacPherson intersection product
\cite{GM1}. As the intersection product is skew symmetric on $I^{\bar m}H_{2k+1}(X;K)$, such a $[z]$ always exists. The error in \cite{GBF21} stems from overlooking that this last fact is not necessarily true in characteristic $2$, where skew symmetric forms and symmetric forms are the same thing and so skew-symmetry does not imply $[z]\cdot [z]=0$.

\paragraph{Corrected computations.}
To begin to remedy the error of \cite{GBF21}, we first observe that it remains true in characteristic $2$ that the map\footnote{Recall from \cite[Corollary 4.3]{GBF21} that the bordism groups depend only on the characteristic of the field, so for characteristic $2$ it suffices to consider $K=\Z_2$.}  $w:\Omega_{4k+2}^{\Z_2-\text{Witt}}\to W(\Z_2)$ is injective, where $W(\Z_2)$ is the Witt group of $\Z_2$ and $w$ takes the bordism class $[X]$ to the class of the intersection form on $I^{\bar m}H_{2k+1}(X;\Z_2)$. For $k>0$, this fact can be proven   as it is proven for $w:\Omega_{4j}^{K-\text{Witt}}\to W(K)$, $j>0$, in \cite{GBF21}: if one assumes that the intersection form on $X$ represents $0$ in $W(\Z_2)$ then  the intersection form is split, in the language of \cite{MH}; see \cite[Corollary III.1.6]{MH}. And so $I^{\bar m}H_{2k+1}(X;\Z_2)$ will possess an isotropic (self-annihilating) element by \cite[Lemma I.6.3]{MH}. The surgery argument can then proceed\footnotemark.   As  $W(\Z_2)\cong \Z_2$ (see \cite[Lemma IV.1.5]{MH}), it follows that $\Omega_{4k+2}^{\Z_2-\text{Witt}}$ is either $0$ or $\Z_2$.

\footnotetext{There is one other possible complication due to characteristic $2$ that must be checked but that does not provide difficulty in the end: For characteristic not equal to $2$, every split form is isomorphic to an orthogonal  sum of hyperbolic planes \cite[Lemma I.6.3]{MH}, and this appears to be used in the proof of Theorem 4.4 of \cite{Si83}, which is heavily referenced in \cite{GBF21}. For characteristic $2$, one can only conclude that a split form is isomorphic to one with matrix $\begin{pmatrix}0&I\\I&A\end{pmatrix}$ for some matrix $A$. However, a detailed reading of the proof of  \cite[Theorem 4.4, particularly page 1097]{Si83} reveals that it is sufficient to have a basis $\{\alpha,\beta,\gamma_1,\ldots,\gamma_{2m}\}$ such that  $\alpha\cdot \alpha=\alpha\cdot \gamma_i=0$ for all $i$ and $\alpha\cdot\beta=1$, and this is certainly provided by a form with the given matrix.}

This argument does not hold for $4k+2=2$ as in this case the dimensions are not sufficient to guarantee that every middle-dimensional intersection homology class is representable by an irreducible element, which is necessary for the surgery argument; see \cite[Lemma 2.2]{Si83}. However, all $2$-dimensional Witt spaces must have at worst isolated singularities, and so in particular such a space must have the form $X\cong (\amalg S_i)/\sim$, where the $S_i$ 
are closed oriented surfaces and the relation $\sim$ glues them together along various isolated points.  But then $X$ is bordant to $\amalg S_i$. This can be seen via a sequence of pinch bordisms as defined by Siegel \cite[Section II]{Si83} that pinch together the regular neighborhoods of sets of points of $\amalg S_i$. To see that the bordism is via a Witt space, it is only necessary to observe that the link of the interior cone point in each such pinch bordism will be a wedge of $S^2$s, and it is easy to compute that $I^{\bar m}H_1(\vee_i S^2;K)=0$ for any $K$. But now, since all closed oriented\footnote{Recall that $\Z_2$-Witt spaces are assumed to be $\Z$-oriented, though see below for more  on orientation considerations} surfaces bound, 
 $\Omega_2^{\Z_2-\text{Witt}}=0$. This special case was also over-looked in \cite{GBF21}, though this argument holds for any field $K$ and is consistent with the claim of \cite{GBF21} that $\Omega_{2}^{K-\text{Witt}}=0$ for all $K$.

Thus we have shown that  $w: \Omega_{4k+2}^{\Z_2-\text{Witt}}\to W(\Z_2)\cong \Z_2$ is an injection for $k\geq 0$, trivially so for $k=0$. Unfortunately, the question  of surjectivity of $w$ in dimensions $4k+2$ is more complicated and not yet fully resolved. We can, however, make the following observation: 
 if $X$ is a $\Z_2$-Witt space of dimension $4k-2$, then\footnote{Recall that the K\"unneth theorem holds within a single perversity when one term is a manifold, so we can compute the intersection forms of such product spaces in the usual way; see e.g. \cite{Ki}.} $w([X\times \C P^2])=w([X])$. So if there is a non-trivial element of $\Omega_{4k-2}^{\Z_2-\text{Witt}}$, then there is a non-trivial element of $\Omega_{4k+2}^{\Z_2-\text{Witt}}$. 

Putting this together with the computations from \cite{GBF21} of $\Omega_{*}^{K-\text{Witt}}$ in dimension $\not\equiv 4k+2\mod 4$   (which remain correct), we have the following theorem:

\begin{theorem}\label{T}
For a field $K$ with $\char(K)=2$, $\Omega_{*}^{K-\text{Witt}}=\Omega_{*}^{\Z_2-\text{Witt}}$, and for\footnote{Since these are geometric bordism groups, they vanish in negative degree.} $k\geq 0$,
\begin{enumerate}
\item $\Omega_{0}^{K-\text{Witt}}\cong\Z$,
\item for $k>0$, $\Omega_{4k}^{K-\text{Witt}}\cong \Z_2$, generated by $[\C P^{2k}]$,
\item $\Omega_{4k+3}^{K-\text{Witt}}=\Omega_{4k+1}^{K-\text{Witt}}=0$,
\item \label{I} 
  Either
\begin{enumerate}
\item $\Omega_{4k+2}^{K-\text{Witt}}=0$ for all $k$, or
\item there exists some $N>0$ such that $\Omega_{4k+2}^{K-\text{Witt}}=0$ for all $k<N$ and $\Omega_{4k+2}^{K-\text{Witt}}\cong \Z_2$ for all $k\geq N$. 
\end{enumerate}
\end{enumerate}
\end{theorem}

We will provide below some further discussion of the difficulties of deciding which case of \eqref{I} holds after discussing unoriented bordism.

\begin{remark}
Independent of the existence or value of $N$ in condition \eqref{I} of the theorem, the computations from \cite[Section 4.5]{GBF21} of $\Omega^{K-\text{Witt}}_*(\,\cdot\,)$ as a generalized homology theory on CW complexes  continue to hold and to imply that for $\char(K)=2$, $$\Omega^{K-\text{Witt}}_n(X)=\Omega^{\Z_2-\text{Witt}}_n(X)\cong \bigoplus_{r+s=n}H_r(X;\Omega^{\Z_2-\text{Witt}}_s).$$ 
\end{remark}

\paragraph{Unoriented bordism.}
Given the motivation to recognize spaces that possess a form of Poincar\'e duality, it seems reasonable to consider $K$-Witt spaces that are \emph{$K$-oriented}. This has no effect when $\char(K)\neq 2$, in which case $K$-orientability is equivalent to $\Z$-orientability as considered in \cite{GBF21}. But when $\char(K)=2$, all pseudomanifolds are $\Z_2$-orientable, which is equivalent to being $K$ orientable, and 
the Poincar\'e duality isomorphism $I^{\bar m}H_k(X;K)\cong \Hom(I^{\bar m}H_{n-k}(X;K),K)$ holds for all such compact pseudomanifolds satisfying the $K$-Witt condition.

If we allow $K$-Witt spaces and $K$-Witt bordism using $K$-orientations, then for  $\char(K)=2$ we are essentially talking about unoriented bordism\footnotemark, so to clarify the notation, let us denote the resulting bordism groups by $\mc N_*^{K-\text{Witt}}$. These groups can be computed as follows:

\footnotetext{One could also define unoriented bordism groups of unoriented compact  PL pseudomanifolds satisfying the $K$-Witt condition with $\char(K)\neq 2$, but it is not clear how to study such groups by the present techniques, as there is no reason to expect that $I^{\bar m}H_*(X;K)$  would satisfy Poincar\'e duality for such a space $X$. }

\begin{theorem}
For a field $K$ with $\char(K)=2$ and for $i\geq 0$,
\begin{equation*}
\mc N_i^{K-\text{Witt}}\cong
\begin{cases}
\Z_2, &i\equiv 0\mod 2,\\
0,  &i\equiv 1\mod 2.
\end{cases}
\end{equation*}
\end{theorem}
Since writing \cite{GBF21}, the author has discovered that this theorem is also provided without detailed proof by Goresky in \cite[page 498]{Go84}. We provide here  the details:
\begin{proof}
 It continues to hold that the local Witt condition depends only on the characteristic of $K$ for the reasons provided in \cite{GBF21}, so we may assume $K=\Z_2$.  To see that $\mc N_n^{\Z_2-\text{Witt}}=0$ for $n$ odd, we simply note that $X$ bounds the closed cone $\bar cX$, which is a $\Z_2$-Witt space. 
The map $w: \mc N_{2k}^{\Z_2-\text{Witt}}\to W(\Z_2)\cong \Z_2$ is onto for each $k>0$, as the intersection pairing  on the $\Z_2$-coefficient middle-dimensional homology of the real projective space $\R P^{2k}$ corresponds to the generator of $W(\Z_2)$ represented by the matrix $\langle 1\rangle$. 
Furthermore,  $w$ is injective for $k>1$ as in the preceding surgery argument, which does not rely on whether or not $X$ is oriented, only on the existence of the intersection pairing over $\Z_2$. In dimension $0$, we have unoriented manifold bordism of points, so $\mc N_0^{\Z_2-\text{Witt}}\cong\Z_2$. Finally, as in the argument above for $\Omega_{2}^{\Z_2-\text{Witt}}$,  the group $\mc N_{2}^{\Z_2-\text{Witt}}$ must be generated by closed surfaces (now not necessarily oriented), so  $\mc N_{2}^{\Z_2-\text{Witt}}$ is a quotient of the unoriented manifold bordism group $\mc N_{2}\cong \Z_2$; thus  $\mc N_{2}^{\Z_2-\text{Witt}}$ must be isomorphic to $\Z_2$ as $w$ maps $\R P^2$  onto the non-trivial element of  $W(\Z_2)\cong \Z_2$. 
\end{proof}

\begin{remark}
 An even simpler version of the argument of \cite{GBF21} implies that as a generalized homology theory
$$\mc N^{K-\text{Witt}}_n(X)\cong \bigoplus_{r+s=n}H_r(X;\mc N^{K-\text{Witt}}_s)$$ for $\char(K)=2$, as in this case one no longer needs a separate argument to handle the odd torsion that can arises in $H_n(X;\Omega^{K-\text{Witt}}_0)$ as a result of $\Omega^{K-\text{Witt}}_0\cong \Z$ not being $2$-primary.
\end{remark}

\paragraph{Further discussion of oriented bordism.}
We next provide some results that demonstrate the difficulty of determining which case of item \eqref{I} of Theorem \ref{T} holds. 

We will first see that $w([M])=0$ for any $\Z$-oriented manifold: Since dimension mod $2$ is the only invariant\footnote{As observed in the proof of \cite[Lemma III.3.3]{MH}, rank mod $2$ yields a homomorphism $W(F)\to \Z_2$ for any field $F$. Since we know that $W(\Z_2)\cong \Z_2$ and that $\langle 1\rangle$, which has rank $1$, is a generator of $W(F)$ (it is certainly non-zero, using \cite[Lemma I.6.3 and Lemma III.1.6]{MH}), it follows that rank mod $2$ determines the isomorphism.  } of $W(\Z_2)$, this is a consequence  of the following lemma, recalling that for a manifold, $I^{\bar m}H_*(M)=H_*(M)$. 

\begin{lemma}
Let $M$ be a closed connected $\Z$-oriented manifold of dimension $4k+2$. Then $\dim(H_{2k+1}(M;\Z_2))\equiv 0\mod 2$.
\end{lemma}
\begin{proof}
By the universal coefficient theorem, $$H_{2k+1}(M;\Z_2)\cong \left(H_{2k+1}(M)\otimes \Z_2\right)\oplus\left( H_{2k}(M)*\Z_2\right),$$ where the asterisk denotes the torsion product.  Let $T_*(M)$ denote the torsion subgroup of $H_*(M)$, and let  $T^2_*(M)$ denote $T_*(M)\otimes \Z_2\cong T_*(M)* \Z_2$; the isomorphism follows from basic homological algebra because $T_*(M)$ is a finite abelian group. $T^2_*(M)$ is a direct sum of $\Z_2$ terms. Then $H_{2k+1}(M)\otimes \Z_2\cong \Z_2^{B} \oplus T^2_{2k+1}(M)$, where $B$ is the $2k+1$ Betti number of $M$, and $H_{2k}(M)*\Z_2\cong T^2_{2k}(M)$. Thus $H_{2k+1}(M;\Z_2)\cong  \Z_2^{B} \oplus T^2_{2k+1}(M)\oplus T^2_{2k}(M)$.
 Since $M$ is a closed $\Z$-oriented manifold, there is a nondegenerate skew-symmetric intersection form on $H_{2k+1}(M;\Q)$, and so $B$ is even.  Since $M$ is a closed $\Z$-oriented manifold, the nonsingular linking pairing $T_{2k+1}(M)\otimes T_{2k}(M)\to \Q/\Z$ gives rise to an isomorphism $T_{2k+1}(M)\cong \Hom(T_{2k}(M),\Q/\Z)$, and since $\Hom(\Z_n,\Q/\Z)\cong \Z_n$, it follows that  $T_{2k+1}(M)\cong T_{2k}(M)$. Therefore $T^2_{2k+1}(M)\cong T^2_{2k}(M)$. Thus $H_{2k+1}(M;\Z/2)$ consists of an even number of $\Z_2$ terms. 
\end{proof}

\begin{remark}
Since the lemma utilizes only integral Poincar\'e duality and the universal coefficient theorem, it follows that, in fact, $w([X])=0$ for any $IP$ space\footnote{Also called ``intersection homology Poincar\'e spaces,'' though this is perhaps a misnomer as ``Poincar\'e spaces'' are generally not required to be manifolds while IP spaces are still expected to be pseudomanifolds.}; these are spaces that satisfy  local conditions guaranteeing that intersection homology Poincar\'e duality holds over the integers and that a universal coefficient theorem holds (see \cite{GS83, Pa90}). 
\end{remark}

A slightly more elaborate argument demonstrates that it is also not possible to have $w([X])\neq 0$ if $X$ is a $\Z$-oriented $\Z_2$-Witt space with at worst isolated singularities:

\begin{proposition}
Let $X$ be a closed $\Z$-oriented $4k+2$-dimensional $\Z_2$-Witt space with at worst isolated singularities. Then $w([X])=0$. 
\end{proposition}
\begin{proof}
Since $X$ has at worst point singularities, it follows from basic intersection homology calculations (see \cite[Section 6.1]{GM1}) that $I^{\bar m}H_{2k+1}(X;\Z_2)\cong \im(H_{2k+1}(M;\Z_2)\to H_{2k+1}(M,\bd M;\Z_2))$, where $M$ is the  compact $\Z$-oriented  PL $\bd$-manifold  obtained by removing an open regular neighborhood of the singular set of $X$. 
 We will show that if $[z]\in\im(H_{2k+1}(M;\Z_2)\to H_{2k+1}(M,\bd M;\Z_2))$, then the intersection product $[z]\cdot [z]=0$. It follows that the intersection pairing on $I^{\bar m}H_{2k+1}(X;\Z_2)$ is split by \cite[Lemma III.1.1]{MH}, since then there can be no non-trivial anisotropic subspace. This implies that $w([X])=0$ by the definition of the Witt group. 

The following argument that $[z]\cdot [z]=0$ was suggested by ``Martin O'' on the web site MathOverflow \cite{MO11}. By Poincar\'e duality, it suffices to show that $\alpha\cup\alpha=0$, where $\alpha$ is the Poincar\'e dual of $[z]$ in $H^{2k+1}(M,\bd M;\Z_2)$. But now $\alpha\cup\alpha=Sq^{2k+1}\alpha=Sq^1Sq^{2k}\alpha=\beta^* Sq^{2k}\alpha$, where $\beta^*$ is the Bockstein associated with  the sequence $0\to \Z_2\to \Z_4\to \Z_2\to 0$ (see \cite[Section 4.L]{Ha}). In the case at hand, this is the Bockstein $\beta^*: H^{4k+1}(M,\bd M;\Z_2)\to H^{4k+2}(M,\bd M;\Z_2)$. But this map is trivial. To see this, observe that there is a commutative diagram 
\begin{diagram}
H^{4k+1}(M,\bd M;\Z_2)&\rTo^{\beta^*}& H^{4k+2}(M,\bd M;\Z_2)\\
\dTo_\cong&&\dTo_\cong\\
H_1(M;\Z_2)&\rTo^{\beta_*}&H_{0}(M;\Z_2),
\end{diagram}
where $\beta_*$ is the homology Bockstein and the vertical maps are Poincar\'e duality. The existence of this diagram follows as in \cite[Lemma 69.2]{MUNK}.  But now $\beta_*:H_1(M;\Z_2)\to H_{0}(M;\Z_2)$ is trivial, as the standard map $\times 2:H_0(M;\Z_2)\to H_0(M;\Z_4)$ is injective. 
\end{proof}

Hence any candidate to have $w([X])=1$ must have singular set of dimension $>0$ and must not be an IP space. Given that all $K$-Witt spaces for $\char(K)\neq 2$ are $K$-Witt bordant to spaces with at worst isolated singularities \cite{Si83, GBF21}, it is unclear how to proceed to determine whether  $\Z_2$-Witt spaces with $w([X])=1$  exist. One method to prove that they do not would be to try to show ``by hand'' that every $\Z_2$-Witt space is $\Z_2$-Witt bordant to a space with at most isolated singularities, but the only  proof currently known to the author of this fact  for fields of other characteristics utilizes the bordism computations of \cite{Si83, GBF21}.

\providecommand{\bysame}{\leavevmode\hbox to3em{\hrulefill}\thinspace}
\providecommand{\MR}{\relax\ifhmode\unskip\space\fi MR }
\providecommand{\MRhref}[2]{%
  \href{http://www.ams.org/mathscinet-getitem?mr=#1}{#2}
}
\providecommand{\href}[2]{#2}

\end{document}